\documentclass[final, nomarks]{dmtcs-episciences}

\usepackage{amsmath}
\usepackage{amsthm}
\usepackage{amssymb}
\usepackage{tikz}
\usepackage[round]{natbib}
\usepackage[utf8]{inputenc}
\usepackage[english]{babel}

\newtheorem{lemma}{Lemma}[section]
\newtheorem{theorem}[lemma]{Theorem}
\newtheorem{corollary}[lemma]{Corollary}
\newtheorem{proposition}[lemma]{Proposition}
\newtheorem{definition}[lemma]{Definition}
\newtheorem{remark}[lemma]{Remark}
\newtheorem{example}[lemma]{Example}

\DeclareMathOperator{\Var}{Var}
\DeclareMathOperator{\Cov}{Cov}
\DeclareMathOperator{\red}{red}

\def\N{\mathbb{N}}
\def\R{\mathbb{R}}
\def\E{\mathbb{E}}
\def\P{\mathbb{P}}
\def\I{\mathbb{I}}

\author{Lisa Hofer}
\title{A Central Limit Theorem for Vincular Permutation Patterns}
\affiliation{Institute of Mathematics, University of Zurich, Switzerland}

\received{2017-04-17}
\revised{2018-02-23}
\accepted{2018-02-23}

\begin{document}
\publicationdetails{19}{2018}{2}{9}{3269}
\maketitle
\begin{abstract}
We study the number of occurrences of any fixed vincular permutation pattern. We show that this statistics on uniform random permutations is asymptotically normal and describe the speed of convergence. To prove this central limit theorem, we use the method of dependency graphs. The main difficulty is then to estimate the variance of our statistics. We need a lower bound on the variance, for which we introduce a recursive technique based on the law of total variance.
\end{abstract}

\section{Introduction}
Permutation patterns are well studied objects in combinatorics and, more recently, also in probability theory. In combinatorics, most of the research is focused on pattern avoidance, \emph{i.e$.$} finding the number of permutations with no occurrences of a given pattern, see \cite{BonaBook,Kitaev}. Another problem is the study of statistics and their distribution. Those objects are studied combinatorially, for instance with multivariate generating functions as in \cite{ChatterjeeDiaconis,CraneDeSalvoElizalde}, but also in probability theory. One statistics of interest is the number of occurrences of a given pattern in a random permutation, where various distributions on permutations are considered. Often, people study the case of uniform permutations, as in \cite{Bona1,Fulman,JansonNakamuraZeilberger}. Among patterns, descents and inversions are the most well-known, see \cite{Fulman}.

The problem we consider in this article is the following: for a fixed pattern $\pi$, what is the asymptotic behaviour of the number of occurrences of $\pi$ in a uniform random permutation $\sigma_n$ of size $n$ going to infinity? We consider this problem when $\pi$ is a so-called vincular pattern.

\bigskip

To explain the precise meaning of this problem, we begin by describing different types of permutation patterns. When we study patterns, permutations are represented in one-line form, \emph{i.e$.$} as a sequence. A permutation of size $n$ is a reordering of the monotone sequence $12\dots n$. The study of patterns is the study of subsequences and their order. For example, consider the permutation
\begin{align*}
\sigma=\boldsymbol{5}8\boldsymbol{2}13\boldsymbol{4}76
\end{align*}
and its subsequence $524$. The unique permutation whose elements are listed in the same relative order is $312$. We say that the subsequence $524$ is an occurrence of the \emph{classical} pattern $312$. Occurrences of classical patterns can be any subsequences. Additional constraints on the subsequences lead to different types of patterns: \emph{tight}, \emph{very tight}, \emph{vincular} and \emph{bivincular} patterns (terminology from \cite{BonaBook,Kitaev}). To count as an occurrence of such a type of pattern, a subsequence must fulfil the constraints listed below:
\begin{itemize}
	\item Tight: all entries have adjacent positions.
	\item Very tight: all entries have adjacent positions and values.
	\item Vincular: some entries have adjacent positions or in other words appear in blocks.
	\item Bivincular: some entries have adjacent positions and some (maybe different) entries have adjacent values.
\end{itemize}
In the literature, patterns are sometimes called consecutive instead of tight and generalized or dashed instead of vincular. Note that vincular patterns generalize both classical and tight patterns. 

As an example of a vincular pattern, consider $3\underline{12}$ where the underlined symbols indicate that the last two entries are required to be in adjacent positions. In the permutation $\sigma$ above, the subsequence $524$ is therefore not an occurrence of this vincular pattern, but the following are: $513$, $534$, $813$, $834$, $847$. The number of occurrences of $3\underline{12}$ in $\sigma$ is $5$. We say the pattern $3\underline{12}$ has two \emph{blocks}, the first is one isolated entry and the second consists in two adjacent entries. Formally, we write a vincular pattern as a pair $(\pi,A)$, where $\pi$ is a permutation giving the order constraint and $A$ gives the required adjacencies.

\bigskip

Instead of counting occurrences of a pattern in a deterministic permutation $\sigma$ as above, we look at $\sigma_n$ which is a uniform permutation of size $n$. Considering a random variable counting the number of occurrences of a fixed pattern of size $k$ in $\sigma_n$, we ask how it behaves asymptotically, as $n$ goes to infinity. Since there are different types of patterns, there is actually a whole family of such problems.

The following answers to these problems are known.
\begin{itemize}
	\item In \cite{Fulman}, J. Fulman proves asymptotic normality for inversions and descents (classical and tight patterns of size $k=2$). In addition, he provides a rate of convergence.
	\item M. B\'ona establishes in \cite{Bona1} asymptotic normality for classical and tight patterns which are monotone (increasing or decreasing).
	\item In \cite{JansonNakamuraZeilberger}, asymptotic normality is shown for all classical patterns. In fact, the authors also establish the joint convergence.
	\item In \cite[Section 8]{CraneDeSalvoElizalde}, asymptotic normality and a rate of convergence are provided for tight patterns of size $k\geq 3$ in random permutations distributed with the so-called Mallows measure of parameter $q$
	(for $q=1$, this measure specializes to the uniform distribution which is of interest in this article). These results are obtained under the assumption that the highest degree component of the variance does not vanish. Proving this kind of variance estimate is often difficult, as we will discuss later in the introduction.
	\item L. Goldstein provides in \cite[Example 3.2]{Goldstein} a rate of convergence for tight patterns in case of the uniform measure. As in \cite{CraneDeSalvoElizalde}, this rate depends on the variance. In contrast to \cite{CraneDeSalvoElizalde}, the method used to obtain this rate applies to a larger family of statistics.
	\item Very tight patterns behave differently: for $k=2$, they are asymptotically Poisson distributed, see \cite{CorteelLouchardPemantle,Kaplansky}. It is easy to see that for $k>2$, the probability to find such a pattern tends to zero (see \cite[p.~3--4]{CorteelLouchardPemantle}).
	\item In \cite{CraneDeSalvo}, Poisson limit theorems are also obtained for tight patterns in Mallows permutations of size $n$, if the parameter of the Mallows distribution $q(n)$ is a function of $n$ of a specific form or if $q$ is fixed but the size of the pattern tends to infinity. This setting is quite orthogonal to the one of other papers (including this one).
\end{itemize}

\bigskip

In the present article, we generalize the result of asymptotic normality to vincular patterns and we also describe the speed of convergence. Our main result, proved in Section \ref{sec_CLT}, is the following.

Let $Z$ denote a standard normal random variable and $d_K$ denote the Kolmogorov distance which is the maximum distance between distribution functions. Denoting $\overline{X^{(\pi,A)}(\sigma_n)}$ the renormalized (mean $0$, variance $1$) random variable which counts the number of occurrences of $(\pi,A)$ in $\sigma_n$, it holds that for some positive constant $C$,
\begin{align*}
d_K\Bigr(\overline{X^{(\pi,A)}(\sigma_n)},Z\Bigr)\leq Cn^{-1/2}.
\end{align*}
This implies immediately that
\begin{align*}
\overline{X^{(\pi,A)}(\sigma_n)}\overset{d}{\rightarrow} Z,
\end{align*}
where $\overset{d}{\rightarrow}$ denotes convergence in distribution, hence proving asymptotic normality. In addition, the bound on $d_K$ quantifies the accuracy of the approximation of $\overline{X^{(\pi,A)}(\sigma_n)}$ by $Z$. Note that this result encompasses the results from \cite{Bona1,CraneDeSalvoElizalde,Fulman,Goldstein,JansonNakamuraZeilberger} previously mentioned except for the joint convergence in \cite{JansonNakamuraZeilberger} and the case of a general parameter $q\neq 1$ in \cite{CraneDeSalvoElizalde}.

\bigskip

Let us now discuss the method of proof. In the literature, the following methods have been used for normal approximation:
\begin{itemize}
	\item \emph{$U$-statistics} in \cite{JansonNakamuraZeilberger}. However, the number of occurrences of vincular patterns is not a $U$-statistics (unlike for classical patterns).
	\item \emph{Exchangeable Stein pairs} in \cite{Fulman} for patterns of size $k=2$. Here, we did not succeed in finding such a pair for patterns of any size.
	\item \emph{Size-bias couplings} in \cite{Goldstein} for tight patterns. We are not aware of such a coupling for vincular patterns.
	\item \emph{Dependency graphs} in \cite{Bona1} and \cite{CraneDeSalvoElizalde}. Such a graph captures the dependencies in a family of random variables. This is useful for our problem since $X^{(\pi,A)}(\sigma_n)$ can be decomposed as a sum of partially dependent random variables (see Equation \eqref{eq_sumdecomp}, p.~\pageref{eq_sumdecomp}).
\end{itemize}
The last three methods are based on \emph{Stein's method} (except for dependency graphs in \cite{Bona1}). This method is used to prove convergence in distribution as well as to describe the approximation error, see Section \ref{sec_Stein} for more details.

In Section \ref{sec_CLT}, we present two approaches to bound the Kolmogorov distance both based on dependency graphs: one using the Stein machinery following \cite{ChenRoellin, Ross} and one using the moment method following \cite{FerayNikeghbali, Janson, Saulis}. While their application is easy, there is one difficulty: we need a lower bound on the variance of $X^{(\pi,A)}(\sigma_n)$ to prove that $d_K(\overline{X^{(\pi,A)}(\sigma_n)},Z)$ goes to $0$.

\bigskip

The method to find that lower bound is discussed in Section \ref{sec_Lbound}. First, we show that $\Var(X^{(\pi,A)}(\sigma_n))$ is a polynomial in $n$. Denoting $j$ the number of blocks of the vincular pattern $(\pi,A)$, the polynomiality implies that 
\begin{align*}
\Var(X^{(\pi,A)}(\sigma_n))=Cn^{2j-1}+\mathcal{O}(n^{2j-2}), \text{ with } C\geq 0.
\end{align*}
If we can show that $C$ is different from $0$, then $\frac{1}{2}Cn^{2j-1}$ is a sharp lower bound (for $n$ big enough). The most natural approach to prove that $C$ is larger than $0$ is to find a formula for $C$ by expressing the variance in terms of covariances (see Equation \eqref{eq_var}, p.~\pageref{eq_var}). This is B\'ona's approach in \cite{Bona1}. Such a formula for $C$ is a \emph{signed} sum of binomials, which in our case is hard to examine. Instead, we introduce a new technique: a recurrence based on the law of total variance. It provides a lower bound for $\Var(X^{(\pi,A)}(\sigma_n))$ of the form $C'n^{2j-3/2}$. Thanks to the polynomiality, this is enough to prove that $C$ is larger than $0$ (see Section \ref{sec_Lbound}).

\bigskip

We conclude this introduction with further directions of research.
\begin{itemize}
	\item In this article, we establish the asymptotic normality for vincular patterns. Moreover, the asymptotic behaviour of very tight patterns is characterized in \cite{CorteelLouchardPemantle}: as Poisson distributed ($k=2$) or rare ($k>2$). Both these types are contained in the larger class of bivincular patterns. One could try to classify such patterns in terms of their limiting distribution.
	
	To study the asymptotically normal case, we suggest using so-called \emph{interaction graphs} or \emph{weighted dependency graphs} introduced in \cite{ChatterjeeDiaconis,Feray}. In \cite{ChatterjeeDiaconis}, interaction graphs are used for a similar statistics, "number of descents plus number of descents in the inverse" which also has constraints in positions and values. The classical tool of dependency graphs does not apply anymore since constraints along these two directions imply that the random variables in a natural sum decomposition are all pairwise dependent (unlike in Equation \eqref{eq_sumdecomp}, p.~\pageref{eq_sumdecomp}). Using the mentioned extensions of dependency graphs avoids this problem, but the main difficulty will again be estimating the variance.
	\item In \cite{JansonNakamuraZeilberger}, joint convergence is established for classical patterns and \cite[Theorem 4.5]{JansonNakamuraZeilberger} describes how dependent the single limit random variables are (for patterns of the same size). It would be interesting to study these questions for vincular patterns. Moreover, one could try to examine the speed of convergence in the multivariate case.
	\item The optimality of the bound on $d_K$ obtained in this article could be investigated. We believe it is optimal since no better bounds are obtained in similar problems (\emph{e.g$.$} subgraph counts in random graphs), but we do not have concrete mathematical evidence for it. 
\end{itemize}

\bigskip

The outline of this article is as follows. In Section \ref{sec_preliminaries}, we give all the necessary background and set up the notation. In Section \ref{sec_CLT}, we show two approaches to prove the main result. Both approaches rely on a lower bound on the variance, which is established in Section \ref{sec_var}.

\section{Background and notation}\label{sec_preliminaries}
Throughout this article, we write a permutation $\sigma$ in one-line notation, \emph{i.e$.$} $\sigma=\sigma_1\sigma_2\dots \sigma_n$. The length of the sequence $\sigma_1\sigma_2\dots \sigma_n$ is the \emph{size} of $\sigma$, denoted by $|\sigma|$. The set of all permutations of size $n$ is denoted by $S_n$. We use $[n]$ to denote the set $\{1,2,\dots,n\}$ and $\binom{[n]}{k}$ for the set of all subsets of $[n]$ which are of size $k$.

\subsection{Vincular patterns}
We refer to \cite{BonaBook} and \cite{Kitaev}, which discuss the notion of permutation patterns. In \cite{BonaBook}, the reader can find information about classical patterns while \cite{Kitaev} discusses patterns in more generality, \emph{e.g$.$} also the case of vincular patterns. A classical pattern is defined as follows.

\begin{definition}
	Let $\sigma\in S_n$ and $\pi\in S_k$. An \emph{occurrence} of the \emph{classical pattern} $\pi$ in $\sigma$ is a subsequence $\sigma_{i_1}\sigma_{i_2}\dots\sigma_{i_k}$ of length $k$ of $\sigma$ such that:
	\begin{itemize}
		\item $\pi_r<\pi_s \iff \sigma_{i_r}<\sigma_{i_s},\text{ for all } r,s \in [k]$.
	\end{itemize}
	We say that $\sigma$ has an occurrence of the classical pattern $\pi$ at positions $i_1,i_2,\dots,i_k$.
\end{definition}

\begin{example}
	Let $\sigma=2374561$. The subsequence $\sigma_2\sigma_5\sigma_7=351$ is an occurrence of the classical pattern $\pi=231$.
\end{example}

In contrast to classical patterns, vincular patterns have additional constraints on the subsequences that are allowed to be counted as an occurrence of the pattern. Certain parts of the pattern are required to be adjacent or, in other words, to appear in blocks.

\begin{definition}
	Let $\sigma\in S_n$, $\pi\in S_k$ and let $A\subseteq [k-1]$. An \emph{occurrence} of the \emph{vincular pattern} $(\pi,A)$ in $\sigma$ is a subsequence $\sigma_{i_1}\sigma_{i_2}\dots\sigma_{i_k}$ of length $k$ of $\sigma$ such that:
	\begin{itemize}
		\item $\sigma_{i_1}\sigma_{i_2}\dots\sigma_{i_k}$ is an occurrence of $\pi$ in the classical sense,
		\item $i_{a+1}=i_{a}+1$, for all $a\in A$.
	\end{itemize}
	We call $A$ the set of \emph{adjacencies}.
\end{definition}

\begin{example}
	Let $\sigma=2374561$. The subsequence $\sigma_2\sigma_3\sigma_7=371$ is an occurrence of the vincular pattern $(\pi,A)=(231,\{1\})$. Note that the subsequence from the previous example, $\sigma_2\sigma_5\sigma_7=351$, is not an occurrence of $(\pi,A)$ since $3$ and $5$ are not adjacent in $\sigma$.
\end{example}

\begin{remark}
	In \cite[Definition 1.3.1]{Kitaev}, the definition of vincular patterns is even more general, allowing also constraints on the beginning and on the end of a pattern occurrence. However, such constraints are rarely considered in the literature, and not included in the present work. The central limit theorem would not hold if we constraint the first entry of the pattern occurrence to be at the beginning, as can be seen in the pattern $12$ with $1$ forced to be the first entry of the permutation.
\end{remark}

In this article, we will work with the above definition but sometimes, it is more convenient to see adjacencies as \emph{blocks}. By block, we mean a maximal subsequence of the pattern whose entries are required to be adjacent. An equivalent way to encode the adjacency information of a vincular pattern is to give a list of block sizes. For example, the vincular pattern $(231,\{1\})$ would be written as $(231,(2,1))$, where the list $(2,1)$ describes a first block of size $2$ followed by a block of size $1$. This idea appears also in \cite[Definition 7.1.2]{Kitaev}, where the list of block sizes is called the \emph{type} of a vincular pattern.

Now, note that the block sizes add up to the size of the pattern, leading to the notion of \emph{composition}. The following definition can be found in \cite[p.~39]{FlajoletSedgewick}.

\begin{definition}
	A \emph{composition} of an integer $n$ is a sequence $(x_1,x_2,\dots,x_\ell)$ of integers, for some $\ell$, such that $n=x_1+x_2+\dots+x_\ell$ and $x_i\geq 1$ for all $i$.
\end{definition}

For example, $(2,1)$ is a composition of $3$. How one can go from one encoding of vincular patterns (by adjacencies or block sizes) to the other is explained by a bijection between subsets of $[k-1]$ and compositions of $k$, where $k=|\pi|$. This bijection associates for instance the composition $(4,2,2,1)$ of size $9$ to $\{1,2,3,5,7\}\subseteq [8]$. The formal construction of the bijection is given below. It will help understanding the rephrasing of the adjacency condition.
\\\\
For $A\subseteq [k-1]$, consider $j=k-|A|$ and $\{c_1,c_2,\dots,c_j\}=[k]\setminus A$ with $c_1<c_2<\dots<c_j=k$, and construct iteratively:
\begin{align}\label{eq_c's}
\begin{cases}
b_1&=c_1,\\ 
b_2&=c_2-c_1,\\
&\vdots\\
b_j&=c_j-c_{j-1}.
\end{cases}
\end{align}
Then, $(b_1,b_2,\dots,b_j)$ is a composition of $k$, since $\sum_{i=1}^{j}b_i=c_j=k$. On the other hand, if $(b_1,b_2,\dots,b_j)$ is a composition of $k$, then the inverse construction is:
\begin{align*}
A=[k]\setminus \Bigr\{\sum_{i=1,\dots,\ell}b_i \Bigr|\; \ell\in [j]\Bigr\},
\end{align*}
which is a subset of $[k-1]$.
\\

So, indeed, a vincular pattern $(\pi,A)$ can be equivalently defined as $(\pi,(b_1,b_2,\dots,b_j))$, where the adjacency condition is rephrased by:
\begin{itemize}
	\item $i_{a+1}=i_{a}+1$, for all $a\in [k]\setminus \bigr\{\sum_{i=1,\dots,\ell}b_i \bigr|\; \ell\in [j]\bigr\}$,
\end{itemize}
or, in other words, the first $b_1$ elements should be adjacent, the next $b_2$ also, and so on. We say the vincular pattern $(\pi,(b_1,b_2,\dots,b_j))$ has $j$ blocks, numbered from left to right, where $b_1,b_2,\dots,b_j$ are the respective block sizes. When we speak about a vincular pattern $(\pi,A)$ with $j$ blocks, it is the above bijection which is underlying.

In the literature, vincular patterns are commonly represented as permutations where some adjacent parts may be underlined, see \cite[Definition 1.3.1]{Kitaev}. What is underlined are the non-trivial blocks of the pattern, \emph{i.e$.$} blocks of size at least $2$. This representation is visual and we will use it when we work with concrete examples. For example, we would write $\underline{23}1$ for the vincular pattern $(231,(2,1))$, or equivalently for $(231,\{1\})$.

\subsection{Vincular pattern statistics on uniform permutations}
Let $\sigma_n$ be a uniform random permutation of size $n$ and let $(\pi,A)$ be a vincular pattern of size $k$. The vincular pattern statistics for the pattern $(\pi,A)$ on $\sigma_n$ is a random variable counting the number of occurrences of $(\pi,A)$ in $\sigma_n$. We denote it by $X^{(\pi,A)}(\sigma_n)$.

Since it is a counting statistics, $X^{(\pi,A)}(\sigma_n)$ can be naturally decomposed as a sum of indicator random variables. First, we introduce a notation to collect all sets of positions that are admissible for occurrences of $(\pi,A)$ in a permutation of size $n$:
\begin{align*}
\mathcal{I}(n,k,A):=\Bigr\{\{i_1,i_2,\dots,i_k\}\in \binom{[n]}{k} \;\Bigr|\; i_{a+1}=i_{a}+1, \text{ for all } a\in A, \text{ where } i_1<i_2<\dots<i_k\Bigr\}.
\end{align*}
Now, the sum decomposition of $X^{(\pi,A)}(\sigma_n)$ is the following:
\begin{align}\label{eq_sumdecomp}
X^{(\pi,A)}(\sigma_n)=\sum_{I\in \mathcal{I}(n,k,A)}X^{\pi}_{I}(\sigma_n),
\end{align}
where $X^{\pi}_{I}(\sigma_n)$ is $1$ if $\sigma_n$ has an occurrence of the (classical) pattern $\pi$ at positions given by $I$ and it is $0$ otherwise.

In \cite{Bona1}, the reader can find this sum decomposition for other types of patterns. The difference lies in the positions over which the summation runs. In our case, the total amount of admissible positions is counted as follows.

\begin{lemma}\label{lemma_card}
	For $n\geq k-j$, where $j=k-|A|$, it holds that:
	\begin{align*}
	|\mathcal{I}(n,k,A)|=\binom{n-k+j}{j}.
	\end{align*}
\end{lemma}

\begin{proof}
	As for patterns, a set of positions $\{i_1,i_2,\dots,i_k\}\in \mathcal{I}(n,k,A)$, with $i_1<i_2<\dots<i_k$, can be split into $j$ blocks using the adjacency information of $A$. These blocks are ordered and of prescribed size. So, essentially, the set $\{i_1,i_2,\dots,i_k\}$ is determined by the set $\{i_1,i_{c_1+1},i_{c_2+1},\dots,i_{c_{j-1}+1}\}$, containing only the first position for each block (see Eq$.$ \eqref{eq_c's} for the description of $c_i$). The trick is to count such sets, but we have to be careful that between the first entries of the blocks there is enough space for the whole blocks. To overcome this problem, we shift everything according to the block sizes, as it is illustrated in Fig$.$ \ref{fig_lemma_card}.
	\begin{figure}[ht]
		\begin{center}
			\begin{tikzpicture}
			\filldraw [black] (0,1) circle (1.5pt);
			\filldraw [black] (0.5,1) circle (1.5pt);
			\filldraw [black] (1,1) circle (1.5pt);
			\filldraw [black] (1.5,1) circle (1.5pt);
			\filldraw [black] (2,1) circle (1.5pt);
			\filldraw [black] (2.5,1) circle (1.5pt);
			\filldraw [black] (3,1) circle (1.5pt);
			\filldraw [black] (3.5,1) circle (1.5pt);
			\filldraw [black] (4,1) circle (1.5pt);
			\filldraw [black] (4.5,1) circle (1.5pt);
			\filldraw [black] (5,1) circle (1.5pt);
			\draw [black] (1,1) circle (5pt);
			\draw [black] (1.5,1) circle (5pt);
			\draw [black] (2,1) circle (5pt);
			\draw [black] (3.5,1) circle (5pt);
			\draw [black] (4,1) circle (5pt);
			\draw [black] (4.5,1) circle (5pt);
			\draw [black] (5,1) circle (5pt);
			\filldraw [black] (0,0.1) circle (1.5pt);
			\filldraw [black] (0.5,0.1) circle (1.5pt);
			\filldraw [black] (1,0.1) circle (1.5pt);
			\draw [black] (1,0.1) circle (5pt);
			\filldraw [black] (2.5,0.1) circle (1.5pt);
			\filldraw [black] (3,0.1) circle (1.5pt);
			\filldraw [black] (3.5,0.1) circle (1.5pt);
			\draw [black] (3.5,0.1) circle (5pt);
			\filldraw [black] (4.5,0.1) circle (1.5pt);
			\draw [black] (4.5,0.1) circle (5pt);
			\draw (0.8,0.7) -- (2.2,0.7);
			\draw (3.3,0.7) -- (4.2,0.7);
			\draw (4.3,0.7) -- (5.2,0.7);
			\draw [->] (1,0.6) -- (1,0.35);
			\draw [->] (3.5,0.6) -- (3.5,0.35);
			\draw [->] (4.5,0.6) -- (4.5,0.35);
			\end{tikzpicture}
		\end{center}
		\caption{Illustration of the bijection.}\label{fig_lemma_card}
	\end{figure}
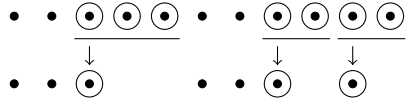
	More precisely, we associate to $\{i_1,i_2,\dots,i_k\}$ a set of positions in $\binom{[n-k+j]}{j}$ where each element $i_{c_t+1}$ is shifted to $i_{c_t+1}-(c_t-t)$. For example, to the set $\{3,4,5,8,9,10,11\}\in \mathcal{I}(11,7,\{1,2,4,6\})$ corresponding to $3$ ordered blocks of sizes $3$, $2$ and $2$, we associate the set $\{3,6,7\}\in \binom{[7]}{3}$, as shown in Fig$.$ \ref{fig_lemma_card}. It is easy to see that this construction describes a bijection between $\mathcal{I}(n,k,A)$ and $\binom{[n-k+j]}{j}$, so that we have:
	\begin{equation*}
	|\mathcal{I}(n,k,A)|=\biggr|\binom{[n-k+j]}{j}\biggr|=\binom{n-k+j}{j}.\\[-5ex]
	\end{equation*}
\end{proof}

\bigskip

\begin{remark}
	The restriction $n\geq k-j$ is necessary, since otherwise the binomial coefficient is not defined, but for $n<k$, we simply have $|\mathcal{I}(n,k,A)|=0$.
\end{remark}

\subsection{Representation of uniform permutations}
We start with the definition of \emph{reduction}, which is given in \cite[p.~1]{JansonNakamuraZeilberger}.

\begin{definition}
	Let $x_1x_2\dots x_n$ be a sequence of $n$ distinct real numbers. The \emph{reduction} of $x_1x_2\dots x_n$, which is denoted by $\red(x_1x_2\dots x_n)$, is the unique permutation $\sigma\in S_n$ such that order relations are preserved, \emph{i.e$.$} $\sigma_i<\sigma_j$ if and only if $x_i<x_j$ for all $i,j\in [n]$.
\end{definition}

As pointed out in \cite[proof of Theorem 4.1]{JansonNakamuraZeilberger}, it is a standard trick to construct a uniform permutation of size $n$ as the reduction of $n$ independent and identically distributed (i.i.d$.$) random variables which are uniform in the interval $[0,1]$. 

\begin{lemma}\label{lemma_red(U1,...,Un)}
	Let $U_1,\dots,U_n$ be i.i.d$.$ random variables, uniform in $[0,1]$. Then, $\red(U_1\dots U_n)$ is a uniform permutation of size $n$.
\end{lemma}

Since we could not find a reference where Lemma \ref{lemma_red(U1,...,Un)} is proved, we provide a short proof here.

\begin{proof}
	 First, the random variables $U_1,\dots,U_n$ are almost surely all distinct. Now, let $\sigma\in S_n$ be arbitrary. Because the random variables $U_i$ are i.i.d$.$, we have for any $\pi \in S_n$,
	\begin{align*}
	\P\big(\red(U_1\dots U_n) = \sigma\big)
	&= \P\big(\red(U_{\pi_1}\dots U_{\pi_n}) = \sigma\big)\\
	&= \P\big(\red(U_1\dots U_n)\pi = \sigma\big)\\
	&= \P\big(\red(U_1\dots U_n) = \sigma\pi^{-1}\big),
	\end{align*}
	showing that $\red(U_1\dots U_n)$ takes all values in $S_n$ with the same probability. So, $\red(U_1\dots U_n)$ is a uniform permutation in $S_n$.
\end{proof}

This representation of a uniform permutation is particularly adapted for our purpose since it relies on \emph{independent} random variables.

\subsection{Dependency graphs}
The dependencies, and not only the pairwise dependencies, within a family of random variables can be represented in a graph as follows, see \cite{BaldiRinott,Bona1,FerayNikeghbali,Rinott}.

\begin{definition}
	A graph $G$ with vertex set $V$ is called a \emph{dependency graph} for a family of random variables $\{X_v\}_{v\in V}$ if it satisfies the following property:
	\begin{quote}
		If $V_1$ and $V_2$ are disjoint subsets of $V$ which are not connected by an edge in $G$, then the sets of random variables $\{X_v\}_{v\in V_1}$ and $\{X_v\}_{v\in V_2}$ are independent.
	\end{quote}
\end{definition}

A family of random variables may have several dependency graphs, capturing sometimes more, sometimes less dependency information. The least information is contained in the complete graph, which is always a dependency graph. As pointed out in \cite{FerayNikeghbali}, dependency graphs are often used to work with sums of \emph{partly} dependent random variables. For example, $X^{(\pi,A)}(\sigma_n)$ decomposed as in Eq$.$ \eqref{eq_sumdecomp} falls in this category. In Section \ref{sec_CLT}, we will see two methods to show asymptotic normality of $X^{(\pi,A)}(\sigma_n)$ using the concept of dependency graphs.  

\subsection{Law of total variance}
The following decomposition formula for the variance can be found in \cite[p.~385--386]{Weiss}.

\begin{proposition}\label{prop_LTV}
	For two random variables $Y$ and $J$, defined on the same probability space, the following holds:
	\begin{align*}
	\Var(Y)=\E[\Var(Y|J)]+\Var(\E[Y|J]).
	\end{align*}
\end{proposition}

This is called the \emph{law of total variance}. It relates the variance of a random variable to its conditional variance and its conditional expectation. The proof uses the \emph{tower property}: $\E[\E[Y|\mathcal{G}]|\mathcal{H}]=\E[Y|\mathcal{H}]$ with $\mathcal{H}=\{\emptyset,\Omega\}$ and $\mathcal{G}=\sigma(J)$, the smallest $\sigma$-algebra such that $J$ is measurable.

It is a natural question if a statement similar to Proposition \ref{prop_LTV} holds for the conditional variance. The equivalent result is called the \emph{law of total conditional variance} and its proof uses the tower property with $\mathcal{H}=\sigma(J_1)$ and $\mathcal{G}=\sigma(J_1,J_2)$.

\begin{proposition}\label{prop_LTCV}
	For three random variables $Y$, $J_1$ and $J_2$, defined on the same probability space, the following holds:
	\begin{align*}
	\Var(Y|J_1)=\E\big[\Var(Y|J_1,J_2)\big|J_1\big]+\Var\big(\E[Y|J_1,J_2]\big|J_1\big).
	\end{align*}
\end{proposition}

A similar expression can be found for $\Var(Y|J_1,J_2)$, the inner conditional variance term in the above equation, and also for more conditioning random variables. Starting from Proposition \ref{prop_LTV} and iteratively using these expressions for the inner conditional variance terms, we can obtain an expression for the variance involving several random variables as conditions.

\begin{proposition}\label{prop_General_LTV}
	For some random variables $Y$ and $J_1,J_2,\dots,J_m$, all defined on the same probability space, the following holds:
	\begin{align*}
	\Var(Y)=\E[\Var(Y|J_1,\dots,J_m)]+\sum_{i=0,\dots,m-1}\E\big[\Var\big(\E[Y|J_1,\dots,J_{i+1}]\big|J_1,\dots,J_i\big)\big].
	\end{align*}
	Observe that the summand for $i=0$ simplifies to $\Var(\E[Y|J_1])$.
\end{proposition}

The only reference we could find for these formulas is \cite[Eqs$.$ 13 and 15]{BowsherSwain}, but it is likely that they appeared independently in other sources. This general decomposition formula for the variance will be very helpful in Section \ref{sec_Lbound}, where we need to find a lower bound for the variance of $X^{(\pi,A)}(\sigma_n)$.

\subsection{Asymptotic notation}
Since we will speak about the asymptotic behaviour of functions, we recall some standard notation.

\begin{definition}\label{def_order}
	Let $f,g:\N\rightarrow\R$ be two functions. We say that:
	\begin{itemize}
		\item $f$ is $\mathcal{O}(g)$, written $f=\mathcal{O}(g)$, if: $\exists\; C>0, n_0 \text{ such that } \forall\; n\geq n_0,\; |f(n)|\leq C|g(n)|,$
		\item $f$ is of \emph{order} $g$, written $f\asymp g$, if: $f=\mathcal{O}(g)$ and $g=\mathcal{O}(f)$,
		\item $f$ is \emph{asymptotically equivalent} to $g$, written $f\sim g$, if: $\underset{n\rightarrow \infty}{\lim}\tfrac{f(n)}{g(n)}=1.$
	\end{itemize}
\end{definition}

\begin{remark}
	\begin{enumerate}
		\item If there exists $C$ such that $f\sim Cn^k$, then $f\asymp n^k$.
		\item Let $p$ be a polynomial in $n$. If $p\asymp n^k$, then there exists $C$ such that $p\sim Cn^k$.
	\end{enumerate}
\end{remark}

\subsection{Stein's method for central limit theorems}\label{sec_Stein}
Stein's method is a technique invented by Charles Stein to bound the distance between two probability distributions. It is used to prove central limit theorems as well as approximation by the Poisson, exponential and other distributions. The survey article \cite{Ross} gives an overview of applications to different distributions and discusses methods to further analyse the bounds. We recall that applications of Stein's method to permutation patterns can be found in \cite{CraneDeSalvo,CraneDeSalvoElizalde,Fulman,Goldstein}. To illustrate the main concept of Stein's method for central limit theorems, we follow \cite[p.~6--9]{Ross}.

First, recall that the \emph{Kolmogorov distance} is a probability metric defined as follows, see \cite[p.~9]{ChenRoellin}, \cite[p.~5]{Ross}.

\begin{definition}\label{def_d_K}
	For two $\R$-valued random variables $X$ and $Y$, let $F_X$ and $F_Y$ be their distribution functions. The Kolmogorov distance between $X$ and $Y$ is defined as:
	\begin{align*}
	d_K(X,Y)=\sup_{t\in \R}|F_X(t)-F_Y(t)|.
	\end{align*} 
\end{definition}

In particular, since convergence of the distribution functions in all continuity points implies convergence in distribution (denoted $\overset{d}{\rightarrow}$), we have the following.

\begin{lemma}\label{lemma_d-conv}
	Let $(X_n)_{n\in \N}$ be a sequence of $\R$-valued random variables and let $Y$ be some $\R$-valued random variable. Then:
	\begin{align*}
	d_K(X_n,Y)\underset{n\rightarrow \infty}{\rightarrow}0 \implies X_n\overset{d}{\rightarrow}Y.
	\end{align*}
\end{lemma}

In general, Stein's method proves convergence in distribution using a functional equation. The next lemma can be used to prove a central limit theorem (CLT), see \cite[Lemma 2.1]{Ross}.

\begin{lemma}[Stein's Lemma]
	Define the functional operator $\mathcal{A}$ by
	\begin{align*}
	\mathcal{A}f(x)=f'(x)-xf(x).
	\end{align*}
	Then, for a real-valued random variable $X$, the following are equivalent.
	\begin{enumerate}
		\item $X$ has the standard normal distribution.
		\item For all absolutely continuous functions $f:\R\rightarrow\R$ such that $\E[|f'(X)|]<\infty$, it holds that $\E[\mathcal{A}f(X)]=0$.  
	\end{enumerate}
\end{lemma}

The first direction (1) $\implies$ (2) is simply integration by parts. More interesting is the second direction. Let $Z$ be a standard normal random variable, $F_Z$ its distribution function and $\I_a$ the indicator function which is $1$ if and only if condition $a$ holds. In \cite[Lemma 2.2]{Ross} it is shown that for any $t\in \R$, there exists a unique bounded solution $f_t$ of the differential equation
\begin{equation}\label{eq_diffeq}
f'_t(x)-xf_t(x)=\I_{x \le t}-F_Z(t).
\end{equation}
Taking $X$ to be any real-valued random variable, we obtain the equation
\begin{equation}\label{eq_funceq}
|\E[f'_t(X)-Xf_t(X)]|=|F_X(t)-F_Z(t)|.
\end{equation}
Hence, the maximum distance between the distribution functions of $X$ and $Z$ is given by
\begin{align*}
d_K(X,Z)=\sup_{t\in \R}|F_X(t)-F_Z(t)|=\sup_{t\in \R}|\E[f'_t(X)-Xf_t(X)]|.
\end{align*}
By Lemma \ref{lemma_d-conv}, to prove the second direction (2) $\implies$ (1), it is enough that (2) holds for all $f_t$ with $t\in \R$ which are solutions to Eq$.$ \eqref{eq_diffeq}.

To prove a CLT, we prove that $(X_n)_{n\in \N}$ almost satisfies the functional equation in Stein's Lemma. We use Eq$.$ \eqref{eq_funceq} to bound $d_K(X_n,Z)$ which also quantifies the rate of convergence. In practice, Stein's method is useful because there exist various techniques to estimate the quantity $\E[f'_t(X)-Xf_t(X)]$: \emph{dependency graphs}, \emph{exchangeable pairs}, \emph{zero- and size-bias couplings}. Theorem \ref{thm_Stein} used in Approach I relies on the dependency graph method which is often useful if the random variable $X$ is a sum of partially dependent random variables like $X^{(\pi,A)}(\sigma_n)$. However, Stein's method may also be used for random variables with different structures. 

We conclude this short summary of Stein's method with two remarks.

\begin{remark}
	\begin{enumerate}
		\item The \emph{Stein-Chen method} is a version developed for Poisson approximation, see \cite{BarbourHolstJanson,CraneDeSalvo}.
		\item It is sometimes easier to study other probability metrics than the Kolmogorov distance. For normal approximation, often the Wasserstein distance is studied, see \cite{Rinott}. In \cite{CraneDeSalvo}, the total variation distance is studied for Poisson approximation.
	\end{enumerate}
\end{remark}

\section{Central limit theorem}\label{sec_CLT}
In this section, we prove a central limit theorem (CLT) for our random variable $X^{(\pi,A)}(\sigma_n)$, the vincular pattern statistics on uniform permutations. Let $(\pi,A)$ be a fixed vincular pattern and assume $n\geq |\pi|\geq 2$ (the other cases are trivial). We normalize $X^{(\pi,A)}(\sigma_n)$:
\begin{align*}
\overline{X^{(\pi,A)}(\sigma_n)}=\frac{X^{(\pi,A)}(\sigma_n)-e_n}{\sqrt{v_n}},
\end{align*}
where $v_n=\Var(X^{(\pi,A)}(\sigma_n))$, $e_n=\E[X^{(\pi,A)}(\sigma_n)]$. The following theorem is the main result of this article, where $\mathcal{N}(0,1)$ denotes the standard normal distribution.

\begin{theorem}\label{thm_CLT}
	Let $Z\sim \mathcal{N}(0,1)$, let $(\pi,A)$ be a fixed vincular pattern and for any $n$, let $\sigma_n$ be uniform in $S_n$. Then, there exists $C>0,n_0$ such that for $n\geq n_0$:
	\begin{align*}
	d_K\Bigr(\overline{X^{(\pi,A)}(\sigma_n)},Z\Bigr)\leq Cn^{-1/2}.
	\end{align*}
	Consequently, it holds that:
	\begin{align*}
	\overline{X^{(\pi,A)}(\sigma_n)}\overset{d}{\rightarrow} Z.
	\end{align*}
\end{theorem}

\begin{remark}
	The second claim of Theorem \ref{thm_CLT} follows from the first claim and Lemma \ref{lemma_d-conv}.
\end{remark}

Note that Theorem \ref{thm_CLT} is not only a CLT result. It also contains information about the speed of convergence, measured in the metric $d_K$. We present two different approaches to prove Theorem \ref{thm_CLT}. One works with Stein's method and the other one with cumulants, but both have in common that they use dependency graphs. Before we start the two different proofs, we give a dependency graph for our problem.

\subsection{Dependency graph for the $X^{\pi}_{I}(\sigma_n)$'s}
We need the following observation about the dependencies between the $X^{\pi}_{I}(\sigma_n)$'s from the sum decomposition of $X^{(\pi,A)}(\sigma_n)$ (see Eq$.$ \eqref{eq_sumdecomp}, p.~\pageref{eq_sumdecomp}). In \cite[Lemma 5.3]{CraneDeSalvo}, this is called the property of \emph{dissociation}.

\begin{lemma}\label{lemma_indep1}
	Let $(\pi,A)$ be a vincular pattern with $|\pi|=k$ and let $\mathcal{F}_1,\mathcal{F}_2\subseteq \mathcal{I}(n,k,A)$. If $\mathcal{F}_1$ and $\mathcal{F}_2$ are such that $\bigcup_{I\in \mathcal{F}_1}I$ and $\bigcup_{I\in \mathcal{F}_2}I$ are disjoint, then the corresponding families of indicators, $\{X^{\pi}_{I}(\sigma_n)\}_{I\in \mathcal{F}_1}$ and $\{X^{\pi}_{I}(\sigma_n)\}_{I\in \mathcal{F}_2}$, are independent.
\end{lemma}

\begin{proof}
	Let $U_1,\dots,U_n$ be independent and uniform in $[0,1]$. By Lemma \ref{lemma_red(U1,...,Un)}, we can represent a uniform permutation as $\red(U_1\dots U_n)$. Since pattern occurrence in non-intersecting subsequences depends on disjoint subsets of the set $\{U_1,\dots,U_n\}$, the independence of these subsets proves the independence of the corresponding families of indicators.
\end{proof}

With the help of Lemma \ref{lemma_indep1}, we can now construct a dependency graph for the family of indicator random variables $\{X^{\pi}_{I}(\sigma_n)\}_{I\in \mathcal{I}(n,k,A)}$, where $k=|\pi|$. We define its vertex set $V_n$ and its edge set $E_n$ as follows:
\begin{align}\label{eq_depgraph}
\begin{cases}
V_n=\mathcal{I}(n,k,A),\\
E_n=\{\{I_1,I_2\}\subseteq \mathcal{I}(n,k,A) |\; I_1\neq I_2,\; I_1\cap I_2 \neq \emptyset \}.
\end{cases} 
\end{align}
If $\mathcal{F}_1$ and $\mathcal{F}_2$ are disjoint subsets of $\mathcal{I}(n,k,A)$ that are not connected by an edge in the graph, then by construction, Lemma \ref{lemma_indep1} applies to $\mathcal{F}_1$ and $\mathcal{F}_2$, ensuring that $\{X^{\pi}_{I}(\sigma_n)\}_{I\in \mathcal{F}_1}$ and $\{X^{\pi}_{I}(\sigma_n)\}_{I\in \mathcal{F}_2}$ are independent. Hence, the dependency graph condition is fulfilled. 

An important parameter is the maximal degree of the dependency graph. 

\begin{lemma}\label{lemma_maxdegree}
	For a fixed vincular pattern $(\pi,A)$ with $j$ blocks, let $D-1$ be the maximal degree of the dependency graph given by \eqref{eq_depgraph}. Then:
	\begin{align*}
	D\asymp n^{j-1}.
	\end{align*}
\end{lemma}

\begin{proof}
	Let $k=|\pi|$. In the proof of Lemma \ref{lemma_card}, we have counted $|\mathcal{I}(n,k,A)|$. It is of order $n^j$. Now, for any fixed vertex $H\in \mathcal{I}(n,k,A)$, consider the quantity $|\{I\in \mathcal{I}(n,k,A)|\; I\cap H\neq \emptyset\}|$. It can be bounded from above and from below as follows:
	\begin{align*}
	\max_{h\in H}|\{I\in \mathcal{I}(n,k,A)|\; h\in I\}|
	&\leq |\{I\in \mathcal{I}(n,k,A)|\; I\cap H\neq \emptyset\}|\\
	&\leq \sum_{h\in H}|\{I\in \mathcal{I}(n,k,A)|\; h\in I\}|,
	\end{align*}
	with
	\begin{align*}
	|\{I\in \mathcal{I}(n,k,A)|\; h\in I\}|=\sum_{\ell\in [j]}|\{I\in \mathcal{I}(n,k,A)|\; h\text{ is in the } \ell\text{-th block of } I\}|.
	\end{align*}
	The quantity $|\{I\in \mathcal{I}(n,k,A)|\; h\text{ is in the } \ell\text{-th block of } I\}|$ is counted similarly to $|\mathcal{I}(n,k,A)|$. The constraint "$h$ is in the $\ell$-th block of $I$" means that $i_{c_\ell+1}$ may admit only a \emph{finite} set of values (see proof of Lemma \ref{lemma_card}, setting $c_0=0$). Compared to $|\mathcal{I}(n,k,A)|$, not $j$ but only $j-1$ blocks are free which decreases the order from $n^j$ to $n^{j-1}$. Since the number of terms in the two sums is independent of $n$, we have:
	\begin{equation*}
	D=\max_{H\in \mathcal{I}(n,k,A)}|\{I\in \mathcal{I}(n,k,A)|\; I\neq H,\; I\cap H\neq \emptyset\}|+1\asymp n^{j-1}.\\[-5ex]
	\end{equation*}
\end{proof}

\bigskip

The dependency graph we just constructed will be used in the next two sections to prove Theorem \ref{thm_CLT}.

\subsection{Approach I: Dependency graphs and Stein's method}
The following theorem can be obtained from results of \cite{ChenRoellin} and \cite{Ross}.

\begin{theorem}\label{thm_Stein}
	Let $Z\sim \mathcal{N}(0,1)$. Let $G$ be a dependency graph for $\{X_i\}_{i=1}^{N}$ and $D-1$ be the maximal degree of $G$. Assume there is a constant $B>0$ such that $|X_i-\E[X_i]|\leq B$ for all $i$. Then, for $W=\frac{1}{\sigma}\sum_{i=1}^{N}(X_i-\E[X_i])$, where $\sigma^2$ is the variance of the sum, it holds that:
	\begin{align*}
	d_K(W,Z)\leq \frac{8B^2D^{3/2}N^{1/2}}{\sigma^2}+\frac{8B^3D^2N}{\sigma^3}.
	\end{align*}
	In particular, if $\sigma^2\asymp B^2DN$ (or $\sigma^2\geq CB^2DN$ for some constant $C>0$), then:
	\begin{align*}
	d_K(W,Z)=\mathcal{O}\biggr(\sqrt{\frac{D}{N}}\biggr).
	\end{align*}
\end{theorem}

\begin{proof}
	Without loss of generality, we assume that $\E[X_i]=0$ for all $i$. Let
	\begin{align*}
	A_i=\{i\}\cup \{1\leq j\leq N |\; \text{vertices }i\text{ and }j\text{ are connected in }G\}.
	\end{align*}
	Clearly, $|A_i|\leq D$ for all $i$. Using \cite[Construction 2B]{ChenRoellin}, we obtain a Stein coupling for $W$ so that with \cite[Corollary 2.6]{ChenRoellin} (for $\alpha=\frac{NB}{\sigma}$, $\beta=\frac{DB}{\sigma}$), we have:
	\begin{align}\label{eq_bound1}
	d_K(W,Z)\leq \frac{2}{\sigma^2}\sqrt{\Var\biggr(\sum_{i=1}^{n}\sum_{j\in A_i}X_iX_j\biggr)}+\frac{8B^3D^2N}{\sigma^3}.
	\end{align}
	Under the assumption of a dependency graph, from the end of the proof of \cite[Theorem 3.5]{Ross}, we have that:
	\begin{align}\label{eq_bound2}
	\Var\biggr(\sum_{i=1}^{N}\sum_{j\in A_i}X_iX_j\biggr)\leq 13D^3\sum_{i=1}^{N}\E[X_i^4]\leq 13B^4D^3N.
	\end{align}
	The final bound follows from Eqs$.$ \eqref{eq_bound1} and \eqref{eq_bound2}.
\end{proof}

We apply Theorem \ref{thm_Stein} to our problem.

\begin{proof}[of Theorem \ref{thm_CLT} (variant I)]
	Let $(\pi,A)$ have $j$ blocks and let $k=|\pi|$. Consider the family of random variables $\{X^{\pi}_{I}(\sigma_n)\}_{I\in \mathcal{I}(n,k,A)}$ and the dependency graph constructed for it in \eqref{eq_depgraph}. Denote by $N$ the size of the family and denote by $D-1$ the maximal degree of the dependency graph. Set $v_n=\Var(X^{(\pi,A)}(\sigma_n))$. From Lemma \ref{lemma_card} and Lemma \ref{lemma_maxdegree}, we have:
	\begin{align*}
	N\asymp n^j,\quad
	D\asymp n^{j-1}.
	\end{align*}
	We will see in Theorem \ref{thm_varasymp} that we have:
	\begin{align*}
	v_n\asymp n^{2j-1}.
	\end{align*}
	The proof being technical, it is postponed to Section \ref{sec_var}.
	Clearly, it holds that $v_n\asymp DN$. Moreover, for all $I\in \mathcal{I}(n,k,A)$, we have:
	\begin{align*}
	|X^{\pi}_I(\sigma_n)-\E[X^{\pi}_I(\sigma_n)]|\leq 1.
	\end{align*}
	Using Theorem \ref{thm_Stein} with $B=1$, we obtain:
	\begin{equation*}
	d_K\Bigr(\overline{X^{(\pi,A)}(\sigma_n)},Z\Bigr)=\mathcal{O}\biggr(\sqrt{\frac{n^{j-1}}{n^j}}\biggr)=\mathcal{O}(n^{-1/2}).\\[-5ex]
	\end{equation*} 
\end{proof}

\bigskip

\begin{remark}
	Almost the same bound as in Theorem \ref{thm_Stein} (giving also a $\mathcal{O}(n^{-1/2})$ in our case) has been obtained by Y. Rinott \cite[Theorem 2.2]{Rinott}.
\end{remark}

\begin{remark}
	Using \cite[Theorem 3.5]{Ross}, the same bound as in Theorem \ref{thm_CLT} can be obtained for the Wasserstein distance $d_W$, which is defined in \cite[p.~5]{Ross}. The bound on $d_W$ would then give a bound on $d_K$, however not an equally good one as Theorem \ref{thm_CLT} ($n^{-1/4}$ instead of $n^{-1/2}$) .
\end{remark}

\subsection{Approach II: Dependency graphs and cumulants}
For any random variable $X$, denote by $\kappa^{(r)}(X)$ its $r$-th cumulant. As in \cite[p.~16]{Saulis}, we say that $X$ satisfies condition $(S_{\gamma,\Delta})$ for some $\gamma\geq 0$, $\Delta>0$ if:
\begin{align*}
|\kappa^{(r)}(X)|\leq \frac{(r!)^{1+\gamma}}{\Delta^{r-2}}, \text{ for all } r\geq 3.
\end{align*}

The following result can be found in \cite[Corollary 2.1]{Saulis}. 

\begin{theorem}\label{thm_Saulis}
	Let $Z\sim \mathcal{N}(0,1)$. For any random variable $X$ satisfying condition $(S_{\gamma,\Delta})$, it holds that:
	\begin{align*}
	d_K(X,Z)\leq \frac{108}{\biggr(\Delta\frac{\sqrt{2}}{6}\biggr)^{\frac{1}{1+2\gamma}}}.
	\end{align*}
\end{theorem}

To prove that condition $(S_{\gamma,\Delta})$ is satisfied by our random variable $X^{(\pi,A)}(\sigma_n)$, we use the following result from \cite[p.~71]{FerayNikeghbali} giving a bound on cumulants of sums of random variables. A slightly weaker version has been established by S. Janson \cite[Lemma 4]{Janson}, see also \cite[p.~71]{FerayNikeghbali}.

\begin{theorem}\label{thm_Janson&FN}
	Let $\{X_v\}_{v\in V}$ be a family of random variables with dependency graph $G$. Denote by $N$ the number of vertices of $G$ and by $D-1$ the maximal degree of $G$. Assume that the $X_v$'s are uniformly bounded by a constant $B$. Then, if $X=\sum_{v\in V}X_v$, for any integer $r\geq 1$, one has:
	\begin{align*}
	|\kappa^{(r)}(X)|\leq 2^{r-1}r^{r-2}ND^{r-1}B^r.
	\end{align*}
\end{theorem}

With the help of Theorem \ref{thm_Janson&FN}, we apply Theorem \ref{thm_Saulis} to our problem.

\begin{proof}[of Theorem \ref{thm_CLT} (variant II)]
	Let $(\pi,A)$ have $j$ blocks and let $k=|\pi|$. Consider the normalized indicator random variables $X^{\pi}_{I}(\sigma_n)$:
	\begin{align*}
	\overline{X^{\pi}_{I}(\sigma_n)}=\frac{X^{\pi}_{I}(\sigma_n)-e_I}{\sqrt{v_n}},
	\end{align*}
	where $v_n=\Var(X^{(\pi,A)}(\sigma_n))$, $e_I=\E[X_I^{\pi}(\sigma_n)]$. Clearly, $\overline{X^{(\pi,A)}(\sigma_n)}$ is the sum of the $\overline{X^{\pi}_{I}(\sigma_n)}$'s. It is easy to see that the dependency graph constructed in \eqref{eq_depgraph} is also a dependency graph for the family $\{\overline{X^{\pi}_{I}(\sigma_n)}\}_{I\in \mathcal{I}(n,k,A)}$. Denote by $N$ its number of vertices and denote by $D-1$ its maximal degree. By Lemma \ref{lemma_card} and Lemma \ref{lemma_maxdegree}, there exist $C_1,C_2>0$ such that:
	\begin{align*}
	N\leq C_1n^j,\quad
	D\leq C_2n^{j-1}.
	\end{align*}
	By Theorem \ref{thm_varasymp} (whose proof is postponed), we have $v_n\asymp n^{2j-1}$. Since $|X^{\pi}_{I}(\sigma_n)-e_I|\leq 1$, there exists $C_3>0$ such that for all $I\in \mathcal{I}(n,k,A)$:
	\begin{align*}
	|\overline{X^{\pi}_I(\sigma_n)}|\leq C_3n^{1/2-j}.
	\end{align*}
	We now use Theorem \ref{thm_Janson&FN} to estimate the cumulants of $\overline{X^{(\pi,A)}(\sigma_n)}$. We use the following simple inequality for factorials: $r^r\leq r!e^r$, valid for $r\geq 1$. For any $r\geq 3$, we obtain that there exists $C>0$ such that:
	\begin{align*}
	\Bigr|\kappa^{(r)}\Bigr(\overline{X^{(\pi,A)}(\sigma_n)}\Bigr)\Bigr|
	&\leq 2^{r-1}r^{r-2}(C_1n^j)(C_2n^{j-1})^{r-1}(C_3n^{1/2-j})^r\\
	&\leq C^rr!n^{-(r-2)/2}\\
	&\leq \frac{r!}{\Delta^{r-2}},
	\end{align*}
	where $\Delta=C^{-3}n^{1/2}$. Here we used that $r/(r-2)\leq 3$ for all $r\geq 3$. Since condition $(S_{\gamma,\Delta})$ is satisfied with $\Delta$ and $\gamma=0$, by Theorem \ref{thm_Saulis}, we obtain: 
	\begin{equation*}
	d_K\Bigr(\overline{X^{(\pi,A)}(\sigma_n)},Z\Bigr)=\mathcal{O}(n^{-1/2}).\\[-5ex]
	\end{equation*} 
\end{proof}

\section{Variance estimate}\label{sec_var}
The main result of this section is the following theorem about the asymptotic behaviour of the variance of $X^{(\pi,A)}(\sigma_n)$ (we still assume $|\pi|\geq 2$). This result has been used in both proofs of Theorem \ref{thm_CLT} given above.

\begin{theorem}\label{thm_varasymp}
	For a vincular pattern $(\pi,A)$ with $j$ blocks and for $\sigma_n$ uniform in $S_n$, there exists $C>0$ such that:
	\begin{align*}
	\Var(X^{(\pi,A)}(\sigma_n))\sim Cn^{2j-1}.
	\end{align*}
\end{theorem}

The proof consists in two steps. First, we show in Section \ref{sec_Ubound} that $\Var(X^{(\pi,A)}(\sigma_n))$ is a polynomial in $n$ of degree at most $2j-1$. This immediately implies that $Cn^{2j-1}$ is an upper bound for the variance (see Corollary \ref{cor_Ubound}). The second step is to find a lower bound of the same form (see Proposition \ref{prop_Lbound}). The lower bound is more important for the CLT result, but it does not follow from the polynomiality. To find it, we present in Section \ref{sec_Lbound} a proof technique building a recurrence from the law of total variance.  

\subsection{Polynomiality and upper bound}\label{sec_Ubound}
Using the sum decomposition of $X^{(\pi,A)}(\sigma_n)$ (see Eq$.$ \eqref{eq_sumdecomp}, p.~\pageref{eq_sumdecomp}) and Lemma \ref{lemma_indep1}, we have:
\begin{align}\label{eq_var}
\Var(X^{(\pi,A)}(\sigma_n))=\sum_{\substack{I,J\in \mathcal{I}(n,k,A): \\ I\cap J\neq \emptyset}}\Cov(X^{\pi}_{I}(\sigma_n), X^{\pi}_{J}(\sigma_n)).
\end{align}
By Lemma \ref{lemma_indep1}, the covariances are $0$ for any $I,J\in \mathcal{I}(n,k,A)$ that do not intersect, explaining the summation index in the above formula.

We use this expression to prove that $\Var(X^{(\pi,A)}(\sigma_n))$ is a polynomial in $n$.
	
\begin{lemma}\label{lemma_poly}
	Let $(\pi,A)$ be a fixed vincular pattern of size $k$ with $j$ blocks and, for $n\geq 1$, let $\sigma_n$ be uniform in $S_n$. Then, for $n\geq 2(k-j)$, $\Var(X^{(\pi,A)}(\sigma_n))$ is a polynomial in $n$ whose degree is at most $2j-1$.
\end{lemma}

\begin{proof}
	The proof idea is to split the sum in Eq$.$ \eqref{eq_var} according to $I,J$ that have the same covariances, and then to count the number of pairs $(I,J)$ in each of these covariance-groups. We will see that the cardinalities of all those groups are polynomials in $n$ and that the number of groups and the covariance values do not depend on $n$. Then, this implies that also $\Var(X^{(\pi,A)}(\sigma_n))$ is a polynomial in $n$.
	
	Let $I,J\in \mathcal{I}(n,k,A)$, $I=\{i_1,i_2,\dots,i_k\}$ with $i_1<i_2<\dots<i_k$ and $J=\{j_1,j_2,\dots,j_k\}$ with $j_1<j_2<\dots<j_k$. To split the sum, we consider $I\cup J=\{u_1,u_2,\dots,u_t\}$ with $u_1<u_2<\dots<u_t$ and $k\leq t\leq 2k$. Note that $t=2k$ if and only if $I\cap J=\emptyset$. Depending on the intersections between $I$ and $J$, $I\cup J$ looks different, see Fig$.$ \ref{fig_lemma_poly}.
	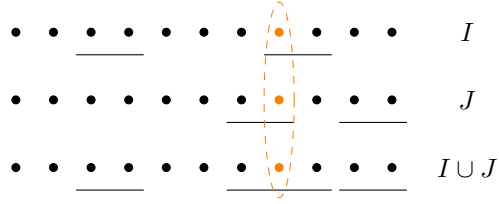
\begin{figure}[ht]
	\begin{center}
		\begin{tikzpicture}
		\filldraw [black] (0,1.9) circle (1.5pt);
		\filldraw [black] (0.5,1.9) circle (1.5pt);
		\filldraw [black] (1,1.9) circle (1.5pt);
		\filldraw [black] (1.5,1.9) circle (1.5pt);
		\filldraw [black] (2,1.9) circle (1.5pt);
		\filldraw [black] (2.5,1.9) circle (1.5pt);
		\filldraw [black] (3,1.9) circle (1.5pt);
		\filldraw [orange] (3.5,1.9) circle (1.5pt);
		\filldraw [black] (4,1.9) circle (1.5pt);
		\filldraw [black] (4.5,1.9) circle (1.5pt);
		\filldraw [black] (5,1.9) circle (1.5pt);
		\filldraw [black] (0,1) circle (1.5pt);
		\filldraw [black] (0.5,1) circle (1.5pt);
		\filldraw [black] (1,1) circle (1.5pt);
		\filldraw [black] (1.5,1) circle (1.5pt);
		\filldraw [black] (2,1) circle (1.5pt);
		\filldraw [black] (2.5,1) circle (1.5pt);
		\filldraw [black] (3,1) circle (1.5pt);
		\filldraw [orange] (3.5,1) circle (1.5pt);
		\filldraw [black] (4,1) circle (1.5pt);
		\filldraw [black] (4.5,1) circle (1.5pt);
		\filldraw [black] (5,1) circle (1.5pt);
		\filldraw [black] (0,0.1) circle (1.5pt);
		\filldraw [black] (0.5,0.1) circle (1.5pt);
		\filldraw [black] (1,0.1) circle (1.5pt);
		\filldraw [black] (1.5,0.1) circle (1.5pt);
		\filldraw [black] (2,0.1) circle (1.5pt);
		\filldraw [black] (2.5,0.1) circle (1.5pt);
		\filldraw [black] (3,0.1) circle (1.5pt);
		\filldraw [orange] (3.5,0.1) circle (1.5pt);
		\filldraw [black] (4,0.1) circle (1.5pt);
		\filldraw [black] (4.5,0.1) circle (1.5pt);
		\filldraw [black] (5,0.1) circle (1.5pt);
		\draw (0.8,1.6) -- (1.7,1.6);
		\draw (3.3,1.6) -- (4.2,1.6);
		\draw (2.8,0.7) -- (3.7,0.7);
		\draw (4.3,0.7) -- (5.2,0.7);
		\draw (0.8,-0.2) -- (1.7,-0.2);
		\draw (2.8,-0.2) -- (4.2,-0.2);
		\draw (4.3,-0.2) -- (5.2,-0.2);
		\draw (6,1.9) node {$I$};
		\draw (6,1) node {$J$};
		\draw (6,0.1) node {$I\cup J$};
		\draw [dashed][orange] (3.5,1) ellipse (0.2cm and 1.3cm);
		\end{tikzpicture}		
	\end{center}
	\caption{Block structure of $I\cup J$.} \label{fig_lemma_poly}
	\end{figure}
	Each $u_s$ in $I\cup J$ comes either from $I$, or from $J$, or from $I\cap J$ here shown in orange. The following function $f_{I,J}:[t]\rightarrow \{I,J,(I,J)\}$ gives to each position $s$ the origin of $u_s$:
	\begin{align*}
	f_{I,J}(s)=
	\left\{
	\begin{array}{ll}
	I & \text{if } u_s\in I \text{ but } u_s\notin J\\
	J & \text{if } u_s\in J \text{ but } u_s\notin I\\
	(I,J) & \text{if } u_s\in I\cap J\\
	\end{array} .
	\right.
	\end{align*}
	Since $\sigma_n$ is uniform, $\Cov(X^{\pi}_{I}(\sigma_n), X^{\pi}_{J}(\sigma_n))$ depends on $I,J$ only through the positions of the intersections of $I$ and $J$. In the example of Fig$.$ \ref{fig_lemma_poly}, $I$ and $J$ intersect at one position only: position $3$ in $I$ and $2$ in $J$. In particular, when two functions $f_{I,J}$ and $f_{I',J'}$ are the same, then the covariances $\Cov(X^{\pi}_{I}(\sigma_n), X^{\pi}_{J}(\sigma_n))$ and $\Cov(X^{\pi}_{I'}(\sigma_n), X^{\pi}_{J'}(\sigma_n))$ are the same.
	We split the sum:
	\begin{align*}
	\sum_{\substack{I,J\in \mathcal{I}(n,k,A): \\ I\cap J\neq \emptyset}}\Cov(X^{\pi}_{I}(\sigma_n), X^{\pi}_{J}(\sigma_n))=\sum_{t=k}^{2k-1}\sum_{f:[t]\rightarrow \{I,J,(I,J)\}}\sum_{\substack{I,J\in \mathcal{I}(n,k,A): \\ |I\cup J|=t,\; f_{I,J}=f}}\Cov(X^{\pi}_{I}(\sigma_n), X^{\pi}_{J}(\sigma_n)).
	\end{align*}
	Each pair $(t,f)$ defines a covariance-group. The number of such pairs and the covariance associated to a pair $(t,f)$ are independent of $n$. To count the size of the groups, we consider the blocks of $I\cup J$. Similarly to patterns, $I$ and $J$ come in blocks. The inherited block structure of $I\cup J$ has three different types of blocks: blocks from $I$, blocks from $J$ and \emph{merged} blocks, coming both from $I$ and $J$. In Fig$.$ \ref{fig_lemma_poly}, the block of size $3$ in $I\cup J$ is a merged block. Instead of blocks, it is equivalent to describe the adjacencies of $I\cup J$. As the blocks, they are inherited from $I$, or from $J$, or both from $I$ and $J$. Neighbouring (as in Fig$.$ \ref{fig_lemma_poly}) or shared adjacencies from $I$ and $J$ take care of the block merges. If $u_s\in I$, then define $r_I(s)$ as the index $\ell$ such that $u_s=i_\ell$ and similarly for $r_J$.
	For a function $f:[t]\rightarrow \{I,J,(I,J)\}$, we define the set $A_f\subseteq [t-1]$ by:
	\begin{align*}
	s\in A_f\iff
	\begin{cases}
	f(s)\in \{I,(I,J)\} \text{ and } r_I(s)\in A,\\
	\text{or } f(s)\in \{J,(I,J)\} \text{ and } r_J(s)\in A.
	\end{cases}
	\end{align*}
	If $I,J\in \mathcal{I}(n,k,A)$ with $|I\cup J|=t$ and $f_{I,J}=f$, then $A_f$ is the set of adjacencies of $I\cup J$. More precisely, then $I\cup J\in \mathcal{I}(n,t,A_f)$. But not every $f$ can occur this way since it has to respect the block structure of $I$ and $J$ given by $A$. For example, $f(s)=I$ or $f(s)=(I,J)$ with $r_I(s)\in A$ but $f(s+1)=J$ for some $s$, is not valid. For a fixed pair $(t,f)$, one of the two following cases occurs:
	\begin{itemize}
		\item There are no $I,J\in \mathcal{I}(n,k,A)$ with $f_{I,J}=f$.
		\item The map $\left\{ \begin{array}{rl}
			\{(I,J)\in \mathcal{I}(n,k,A)^2 |\; f_{I,J}=f\}&\rightarrow \mathcal{I}(n,t,A_f),\\
			(I,J)&\mapsto I\cup J,
		\end{array}
		\right.$ is a bijection.
		\\
		Indeed, if $K\in \mathcal{I}(n,t,A_f)$, then $(I,J)$ can be reconstructed from $K=I\cup J$ and $f=f_{I,J}$. By construction of $A_f$, $I$ and $J$ will be in $\mathcal{I}(n,k,A)$.
	\end{itemize}
	So, the cardinalities of the covariance-groups are either $0$ or given by the cardinality of the corresponding $\mathcal{I}(n,t,A_f)$. By Lemma \ref{lemma_card}, for $n\geq t-(t-|A_f|)=|A_f|$, we have:
	\begin{align*}
	|\mathcal{I}(n,t,A_f)|=\binom{n-t+(t-|A_f|)}{t-|A_f|}=\binom{n-|A_f|}{t-|A_f|}.
	\end{align*}
	This is a polynomial in $n$, since $t$ and $|A_f|$ do not depend on $n$. The maximal value $|A_f|$ can take is $2(k-j)$. So, for $n\geq 2(k-j)$, $\Var(X^{(\pi,A)}(\sigma_n))$ is the sum of polynomials in $n$, which is again a polynomial in $n$. Its maximal degree is the maximal value for $t-|A_f|$ (the number of blocks of $I\cup J$) which is $2j-1$.
\end{proof}

\begin{remark}
	This polynomiality result can be used to compute expressions of the variance for small patterns by polynomial interpolation (whence the desire to be precise on the range of values of $n$ for which the polynomiality holds).
\end{remark}

From Lemma \ref{lemma_poly}, we directly obtain an upper bound for $\Var(X^{(\pi,A)}(\sigma_n))$.

\begin{corollary}\label{cor_Ubound}
	Let $(\pi,A)$ be a fixed vincular pattern of size $k$ with $j$ blocks and let $\sigma_n$ be uniform in $S_n$. Then, there exists $C>0,n_0$ such that for $n\geq n_0$:
	\begin{align*}
	\Var(X^{(\pi,A)}(\sigma_n))\leq Cn^{2j-1}.
	\end{align*}
\end{corollary}

\subsection{Lower bound}\label{sec_Lbound}
Our proof technique for finding a sharp lower bound on $\Var(X^{(\pi,A)}(\sigma_n))$ uses a recurrence that we obtain from the law of total variance. Working directly with the variance decomposition (see Eq$.$ \eqref{eq_var}) would be more difficult since covariances can be negative whereas the law of total variance involves only non-negative terms. We first discuss what conditioning we want to use in the law of total variance. Then, we show how to obtain the recurrence relation. And finally, we deduce a recursive estimation from which we then derive the lower bound.

For the rest of this section, let $(\pi,A)$ be a fixed vincular pattern of size $k$ with $j$ blocks. Let $U_1,U_2,\dots$ be independent and uniform in $[0,1]$. For any $n$, set $\sigma_n=\red(U_1\dots U_n)$. By Lemma \ref{lemma_red(U1,...,Un)}, $\sigma_n$ is uniform in $S_n$. Moreover, for any $n$, we define  $v_n=\Var(X^{(\pi,A)}(\sigma_n))$. For simplicity, we also set $Y=X^{(\pi,A)}(\sigma_n)$.

We apply the general law of total variance (see Proposition \ref{prop_General_LTV}) on $Y$ where we shall condition on the last few entries of $\sigma_n$. The number of these entries is the size of the last block of the pattern $(\pi,A)$, denoted $b_j$. More precisely, we condition on $U_n,U_{n-1},\dots,U_{n-b_j+1}$. We obtain the following expression for $v_n$:
\begin{equation}\label{eq_rec1}
\begin{split}
v_n=\Var(Y)
&=\E[\Var(Y|U_n,\dots,U_{n-b_j+1})]\\
&\hspace{0.5cm}+\sum_{i=0,\dots,b_j-1}\E\big[\Var\big(\E[Y|U_n,\dots,U_{n-i}]\big|U_n,\dots,U_{n-i+1}\big)\big].
\end{split}
\end{equation}

We now turn to the recurrence where we will see why this conditioning is a good choice. We split $Y$ in two parts as follows: $Y=B+C$ with
\begin{align*}
B=\sum_{I\in \mathcal{I}(n,k,A):\; n\notin I}X^{\pi}_{I}(\sigma_n),\quad
C=\sum_{I\in \mathcal{I}(n,k,A):\; n\in I}X^{\pi}_{I}(\sigma_n).
\end{align*}

Observe that $B=X^{(\pi,A)}(\sigma_{n-1})$, since $\sigma_n$ is uniform in $S_n$ and $\sigma_{n-1}$ is uniform in $S_{n-1}$. Applying Proposition \ref{prop_General_LTV} on $B$, with the conditions $U_{n-1},\dots,U_{n-b_j+1}$, we have:
\begin{equation}\label{eq_rec2}
\begin{split}
v_{n-1}=\Var(B)
&=\E[\Var(B|U_{n-1},\dots,U_{n-b_j+1})]\\
&\hspace{0.5cm}+\sum_{i=0,\dots,b_j-2}\E\big[\Var\big(\E[B|U_{n-1},\dots,U_{n-i-1}]\big|U_{n-1},\dots,U_{n-i}\big)\big].
\end{split}
\end{equation}

In the expression for $v_n$, Eq$.$ \eqref{eq_rec1}, we want to recover $v_{n-1}$ from $\E[\Var(Y|U_n,\dots,U_{n-b_j+1})]$. Since $Y=B+C$ and since $B$ is independent of $U_n$, we have:
\begin{equation}\label{eq_rec3}
\begin{split}
\E[\Var(Y|U_n,\dots,U_{n-b_j+1})]
&=\E[\Var(B|U_{n-1},\dots,U_{n-b_j+1})]\\
&\hspace{0.5cm}+\E[\Var(C|U_n,\dots,U_{n-b_j+1})]\\
&\hspace{0.5cm}+2\E[\Cov(B,C|U_n,\dots,U_{n-b_j+1})].
\end{split}
\end{equation}

Because $\E[\Var(B|U_{n-1},\dots,U_{n-b_j+1})]$ also appears in the expression for $v_{n-1}$, using Eqs$.$ \eqref{eq_rec1}, \eqref{eq_rec2} and \eqref{eq_rec3}, we obtain the following recurrence relation:
\begin{equation}\label{eq_recrel}
\begin{split}
v_n-v_{n-1}
&=2\E[\Cov(B,C|U_n,\dots,U_{n-b_j+1})]\\
&\hspace{0.5cm}+\E[\Var(C|U_n,\dots,U_{n-b_j+1})]\\
&\hspace{0.6cm}+\Bigr(\sum_{i=0,\dots,b_j-2}\E\big[\Var\big(\E[Y|U_n,\dots,U_{n-i}]\big|U_n,\dots,U_{n-i+1}\big)\big]\\
&\hspace{2cm}-\sum_{i=0,\dots,b_j-2}\E\big[\Var\big(\E[B|U_{n-1},\dots,U_{n-i-1}]\big|U_{n-1},\dots,U_{n-i}\big)\big]\Bigr)\\
&\hspace{0.5cm}+\E\big[\Var\big(\E[Y|U_n,\dots,U_{n-b_j+1}]\big|U_n,\dots,U_{n-b_j+2}\big)\big].
\end{split}
\end{equation}

The right-hand side of Eq$.$ \eqref{eq_recrel} is grouped in four terms. To find a lower bound with the help of this recurrence relation, we examine all these terms. We will need the following result which can be proved very similarly to Lemma \ref{lemma_indep1}. 

\begin{lemma}\label{lemma_indep2}
	Let $(\pi,A)$ be a vincular pattern of size $k$ with $j$ blocks and a last block of size $b_j$. Let $U_1,\dots,U_n$ be independent and uniform in $[0,1]$ and let $\sigma_n=\red(U_1\dots U_n)$. Then, for $I,J\in \mathcal{I}(n,k,A)$ with $I\cap J\subseteq\{n-b_j+1,\dots,n\}$, conditionally on $U_{n-b_j+1},\dots,U_n$, the random variables $X^{(\pi,A)}_{I}(\sigma_n)$ and $X^{(\pi,A)}_{J}(\sigma_n)$ are independent.
\end{lemma}

Lemma \ref{lemma_indep2} is the reason why we want to condition on the last $b_j$ entries of $\sigma_n$. Moreover, we need the following definition in order to work with sorted sequences.

\begin{definition}
	Let $x_1x_2\dots x_n$ be a sequence of $n$ distinct real numbers. Then, $sort(x_1\dots x_n)$ is the \emph{sorted} sequence which contains $x_1,x_2,\dots,x_n$ but in increasing order.
\end{definition}

We now examine separately the four terms of the recurrence relation (see Eq$.$ \eqref{eq_recrel}). We want to find lower bounds for each of them. The following computations hold for $n$ large enough.

\paragraph{\textbf{First term}} First, we have:
\begin{align*}
\Var(C|U_n,\dots,U_{n-b_j+1})=\sum_{I,J\in \mathcal{I}(n,k,A):\; n\in I\cap J}\Cov(X^{\pi}_{I}(\sigma_n),X^{\pi}_{J}(\sigma_n)|U_n,\dots,U_{n-b_j+1}).
\end{align*}
The constraint $n\in I\cap J$ means that $I$ and $J$ intersect at least in the whole last block. By Lemma \ref{lemma_indep2}, for the above covariances to be non-zero, $I$ and $J$ must intersect at least in one more block. Similar arguments as in the proof of Lemma \ref{lemma_maxdegree} show that the number of non-zero covariances is $\mathcal{O}(n^{2j-3})$. Since the covariances are bounded by $1$, this implies:
\begin{align*}
\Var(C|U_n,\dots,U_{n-b_j+1})=\mathcal{O}(n^{2j-3}).
\end{align*} 
Note that the constant in the $\mathcal{O}$-term does not depend on $U_n,\dots,U_{n-b_j+1}$, which is important when we take the expectation. Then, by the Cauchy-Schwarz inequality, it holds that:
\begin{equation*}
|\Cov(B,C|U_n,\dots,U_{n-b_j+1})|\leq \Var(B|U_n,\dots,U_{n-b_j+1})^{1/2}\mathcal{O}(n^{j-3/2}).
\end{equation*}
And by the Jensen inequality:
\begin{equation*}
\E[|\Cov(B,C|U_n,\dots,U_{n-b_j+1})|]\leq \E[\Var(B|U_n,\dots,U_{n-b_j+1})]^{1/2}\mathcal{O}(n^{j-3/2}).
\end{equation*}
Equation \eqref{eq_rec2} implies that $\E[\Var(B|U_n,\dots,U_{n-b_j+1})]\leq v_{n-1}$, so that:
\begin{equation*}
\E[|\Cov(B,C|U_n,\dots,U_{n-b_j+1})|]\leq v_{n-1}^{1/2}\mathcal{O}(n^{j-3/2}).
\end{equation*}
Finally, there exists $C_1>0$ such that:
\begin{equation}\label{eq_term1}
2\E[\Cov(B,C|U_n,\dots,U_{n-b_j+1})]\geq -C_1n^{j-3/2}v_{n-1}^{1/2}.
\end{equation}

\paragraph{\textbf{Second term}} We will simply use the trivial inequality
\begin{equation}\label{eq_term2}
\E[\Var(C|U_n,\dots,U_{n-b_j+1})]\geq 0.
\end{equation}

\paragraph{\textbf{Third term}} For $0\leq m\leq b_j-1$, define 
\begin{align*}
&A_m=\sum_{\substack{I\in \mathcal{I}(n,k,A): \\ n,\dots,n-m+1\notin I, n-m\in I} }X^{\pi}_{I}(\sigma_n),\\
&A_{b_j}=\sum_{\substack{I\in \mathcal{I}(n,k,A): \\ n,\dots,n-b_j+1\notin I} }X^{\pi}_{I}(\sigma_n).
\end{align*}
Clearly, $Y=A_0+A_1+\dots+A_{b_j}$, while $C=A_0$ and $B=A_1+A_2+\dots+A_{b_j}$. Note that for $0\leq m\leq b_j$, $A_m$ is independent of $U_n,\dots,U_{n-m+1}$. For any $0\leq i\leq b_j-1$, we have:
\begin{equation*}
\E[Y|U_n,\dots,U_{n-i}]=\sum_{m=0}^{i}\E[A_m|U_{n-m},\dots,U_{n-i}]+\sum_{m=i+1}^{b_j}\E[A_m].
\end{equation*}
Taking the variance, the second sum will not contribute since it is deterministic. For the first sum, \emph{i.e$.$} the case $0\leq m\leq i$, we compute $\E[A_m|U_{n-m},\dots,U_{n-i}]$. Before giving the general formula, we consider a simple example.
\\\\
We explain how to obtain a formula for the random variable $\E[A_0|U_{n},U_{n-1}]$. Consider the pattern $\pi=54\underline{231}$. Then, $A_0$ is the sum of the indicators $X^{\pi}_{I}(\sigma_n)$ for $I\in \mathcal{I}(n,5,\{3,4\})$ such that $n\in I$. This implies automatically that $n-2,n-1\in I$ due to the adjacencies of the given pattern. Assume $I=\{i_5,i_4,n-2,n-1,n\}$ with $i_5<i_4<n-2$. Then, we have:
\begin{align*}
\P\big(X^{\pi}_{I}(\sigma_n)=1\big|U_{n},U_{n-1}\big)
&=\P\big(\red(U_{i_5}U_{i_4}U_{n-2}U_{n-1}U_{n})=54231\big|U_{n},U_{n-1}\big)\\
&=\P(U_{n-1}>U_{n-2}>U_{n}|U_{n},U_{n-1})\P(U_{i_5}>U_{i_4}>U_{n-1}|U_{n-1})\I_{U_{n-1}>U_{n}}\\
&=(U_{n-1}-U_{n})\frac{(1-U_{n-1})^2}{2!}\I_{U_{n-1}>U_{n}},
\end{align*}
where $\I_x$ is the indicator function which is $1$ if and only if condition $x$ holds. We used that the $U_i$'s are independent and uniform in $[0,1]$. Since there are $\binom{n-3}{2}$ choices for $i_5$ and $i_4$, we have:
\begin{align*}
\E[A_0|U_{n},U_{n-1}]
&=\sum_{I\in \mathcal{I}(n,5,\{3,4\}):\; n\in I}\P\big(X^{\pi}_{I}(\sigma_n)=1\big|U_{n},U_{n-1}\big)\\
&=\binom{n-3}{2}\Biggr((U_{n-1}-U_{n})\frac{(1-U_{n-1})^2}{2!}\I_{U_{n-1}>U_{n}}\Biggr).
\end{align*}
This is the explicit formula for the random variable $\E[A_0|U_{n},U_{n-1}]$.
\\

For the general formula, let $a_1\dots a_{1+i-m}=sort(\pi_{k-(i-m)}\dots\pi_k)$ (fixed, since $\pi$ is) and define $\ell_1,\dots,\ell_{1+i-m}$ such that $U_{\ell_1}\dots U_{\ell_{1+i-m}}=sort(U_{n-i}\dots U_{n-m})$ (not fixed). Then:
\begin{align*}
\E[A_m|U_{n-m},\dots,U_{n-i}]=B_m(n)\times S(U_{n-m},\dots,U_{n-i}),
\end{align*}
where
\begin{align*}
B_m(n)
&=\binom{n-m-k+j-1}{j-1},
\end{align*}
\begin{align*}
S(U_{n-m},\dots,U_{n-i})
&=\frac{U_{\ell_1}^{a_1-1}(U_{\ell_2}-U_{\ell_1})^{a_2-a_1-1}\cdots (1-U_{\ell_{1+i-m}})^{k-a_{1+i-m}}}{(a_1-1)!(a_2-a_1-1)!\cdots (k-a_{1+i-m})!} \I_{\substack{\red(U_{n-i}\dots U_{n-m}) \\ =\red(\pi_{k-(i-m)}\dots\pi_k)}}.
\end{align*}
The binomial coefficient counts the number of $I\in \mathcal{I}(n,k,A)$ fulfilling the constraints for $A_m$: this set is in bijection with $\mathcal{I}(n-b_j-m,k-b_j,A\cap [k-b_j])$ so that its cardinality is given by Lemma \ref{lemma_card}. For each such $I$, the probability that the corresponding values in $\sigma_n$ are in the good order is given by the fraction. The indicator takes care of the order of the given $U_i$'s.

\bigskip

We need these computations to compare the expectations (of $Y$ and of $B$) appearing in the third term in Eq$.$ \eqref{eq_recrel}. For $Y$, we have:
\begin{equation}\label{eq_term3-}
\begin{split}
&\E\big[\Var\big(\E[Y|U_n,\dots,U_{n-i}]\big|U_n,\dots,U_{n-i+1}\big)\big]\\
&=\E\biggr[\Var\biggr(\sum_{m=0}^{i}B_m(n)S(U_{n-m},\dots,U_{n-i})\biggr|U_n,\dots,U_{n-i+1}\biggr)\biggr]\\
&=n^{2j-2}\E\biggr[\Var\biggr(\sum_{m=0}^{i}\frac{1}{(j-1)!}S(U_{n-m},\dots,U_{n-i})\biggr|U_n,\dots,U_{n-i+1}\biggr)\biggr]+\mathcal{O}(n^{2j-3}).
\end{split}
\end{equation}
And for $B$, we obtain:
\begin{align*}
&\E\big[\Var\big(\E[B|U_{n-1},\dots,U_{n-i-1}]\big|U_{n-1},\dots,U_{n-i}\big)\big]\\
&=\E\biggr[\Var\biggr(\sum_{m=1}^{i+1}B_m(n)S(U_{n-m},\dots,U_{n-i-1})\biggr|U_{n-1},\dots,U_{n-i}\biggr)\biggr]\\
&=n^{2j-2}\E\biggr[\Var\biggr(\sum_{m=1}^{i+1}\frac{1}{(j-1)!}S(U_{n-m},\dots,U_{n-i-1})\biggr|U_{n-1},\dots,U_{n-i}\biggr)\biggr]+\mathcal{O}(n^{2j-3})\\
&=n^{2j-2}\E\biggr[\Var\biggr(\sum_{m=0}^{i}\frac{1}{(j-1)!}S(U_{n-m},\dots,U_{n-i})\biggr|U_n,\dots,U_{n-i+1}\biggr)\biggr]+\mathcal{O}(n^{2j-3}),
\end{align*}
where the last step uses that the $U_i$'s are independent and identically distributed (i.i.d$.$) so that we can replace $(U_{n-1},\dots,U_{n-i-1})$ by $(U_n,\dots,U_{n-i})$. In particular:
\begin{equation}\label{eq_term3}
\begin{split}
&\sum_{i=0,\dots,b_j-2}\E\big[\Var\big(\E[Y|U_n,\dots,U_{n-i}]\big|U_n,\dots,U_{n-i+1}\big)\big]\\
&-\sum_{i=0,\dots,b_j-2}\E\big[\Var\big(\E[B|U_{n-1},\dots,U_{n-i-1}]\big|U_{n-1},\dots,U_{n-i}\big)\big]\\
&=\mathcal{O}(n^{2j-3}).
\end{split}
\end{equation}

\paragraph{\textbf{Fourth term}} From Eq$.$ \eqref{eq_term3-}, we have:
\begin{align*}
&\E\big[\Var\big(\E[Y|U_n,\dots,U_{n-b_j+1}]\big|U_n,\dots,U_{n-b_j+2}\big)\big]=n^{2j-2}\E[f(U_n,\dots,U_{n-b_j+2})]+\mathcal{O}(n^{2j-3}),
\end{align*}
where
\begin{align*}
f(U_n,\dots,U_{n-b_j+2})=\Var\biggr(\sum_{m=0}^{b_j-1}\frac{1}{(j-1)!}S(U_{n-m},\dots,U_{n-b_j+1})\biggr|U_n,\dots,U_{n-b_j+2}\biggr).
\end{align*}
Note that with positive probability, $\red(U_{n-b_j+2}\dots U_{n})=\red(\pi_{k-b_j+2}\dots\pi_k)$. On this event, as a function of $U_{n-b_j+1}$, the random variable $\sum_{m=0,\dots,b_j-1}S(U_{n-m},\dots,U_{n-b_j+1})$ is not a.s. constant (in particular,  $\I_{\red(U_{n-b_j+1}\dots U_{n})=\red(\pi_{k-b_j+1}\dots\pi_k)}$ takes value $1$ or $0$ with positive probability each). So, with positive probability, $f(U_n,\dots,U_{n-b_j+2})$ is non-zero which implies $\E[f(U_n,\dots,U_{n-b_j+2})]>0$. Furthermore, since the $U_i$'s are i.i.d$.$, $\E[f(U_n,\dots,U_{n-b_j+2})]$ does not depend on $n$. Hence, there exists $C_2>0$ such that:
\begin{equation}\label{eq_term4}
\E\big[\Var\big(\E[Y|U_n,\dots,U_{n-b_j+1}]\big|U_n,\dots,U_{n-b_j+2}\big)\big]\geq C_2n^{2j-2}.
\end{equation}  
\bigskip

\paragraph{\textbf{Conclusion}}
Putting Eqs$.$ \eqref{eq_term1}, \eqref{eq_term2}, \eqref{eq_term3} and \eqref{eq_term4} together, we obtain for some $C_3>0$:
\begin{equation}\label{eq_rec5}
\begin{split}
v_n
&\geq v_{n-1}\bigr(1-C_1n^{j-3/2}v_{n-1}^{-1/2}\bigr)+C_2n^{2j-2}+\mathcal{O}(n^{2j-3})\\
&\geq v_{n-1}\bigr(1-C_1n^{j-3/2}v_{n-1}^{-1/2}\bigr)+C_3n^{2j-2}.
\end{split}
\end{equation}

Using first Eq$.$ \eqref{eq_rec1} and then Eq$.$ \eqref{eq_term3-} (where $i=0$ gives $\E[\Var(\E[Y|U_n])]=\Var(\E[Y|U_n])$), we directly obtain:
\begin{align*}
v_n\geq \Var(\E[Y|U_n])=n^{2j-2}\Var\biggr(\frac{1}{(j-1)!}S(U_n)\biggr)+\mathcal{O}(n^{2j-3}),
\end{align*}
where
\begin{align*}
S(U_n)=\frac{U_n^{\pi_k-1}(1-U_n)^{k-\pi_k}}{(\pi_k-1)!(k-\pi_k)!}.
\end{align*}
Because we assume $k\geq 2$, $S(U_n)$ is not a.s. constant. Moreover, its variance does not depend on $n$, so that there exists $C_4>0$ such that:
\begin{equation}\label{eq_rec4}
v_n\geq C_4n^{2j-2}.
\end{equation}

Using Eq$.$ \eqref{eq_rec4} for $v_{n-1}$, we can refine Eq$.$ \eqref{eq_rec5} for some $C_5>0$:
\begin{equation}\label{eq_rec6}
v_n\geq v_{n-1}\bigr(1-C_5n^{-1/2}\bigr)+C_3n^{2j-2}.
\end{equation}

Equation \eqref{eq_rec6} is the recursive estimation that we now use to obtain a lower bound for the variance $v_n=\Var(X^{(\pi,A)}(\sigma_n))$ which is sharper than the one given by Eq$.$ \eqref{eq_rec4}.

\begin{proposition}\label{prop_Lbound}
	Let $(\pi,A)$ be a fixed vincular pattern of size $k$ with $j$ blocks and let $\sigma_n$ be uniform in $S_n$. Then, there exists $C>0,n_0$ such that for $n\geq n_0$:
	\begin{align*}
	\Var(X^{(\pi,A)}(\sigma_n))\geq Cn^{2j-1}.
	\end{align*}
\end{proposition}

\begin{proof}
	For any $n$, let $v_n=\Var(X^{(\pi,A)}(\sigma_n))$. Let $n_1$ be large enough such that $1-\frac{C_5}{\sqrt{n_1}}\geq 0$ and such that Eq$.$ \eqref{eq_rec6} holds for $n\geq n_1$. By Eq$.$ \eqref{eq_rec6}, we recursively obtain:
	\begin{align*}
	v_n \geq \sum_{\ell=n_1}^{n}C_3\ell^{2j-2}\prod_{m=\ell+1}^{n}\Bigr(1-\frac{C_5}{\sqrt{m}}\Bigr).
	\end{align*}
	Since $1-x\geq e^{-2x}\geq 1-2x$ for $x\leq 1$, we have for $\ell$ big enough:
	\begin{align*}
	\prod_{m=\ell+1}^{n}\Bigr(1-\frac{C_5}{\sqrt{m}}\Bigr)
	\geq \prod_{m=\ell+1}^{n}e^{-2\frac{C_5}{\sqrt{m}}}
	=e^{-2\sum_{m=\ell+1}^{n}\frac{C_5}{\sqrt{m}}}
	\geq 1-2\sum_{m=\ell+1}^{n}\frac{C_5}{\sqrt{m}}.
	\end{align*} 
	Hence, there exists $C_6>0$ such that for $n$ big enough:
	\begin{align*}
	v_n
	&\geq \sum_{\ell=n-C_6\sqrt{n}}^{n}C_3\ell^{2j-2}\Bigr(1-2C_5\sum_{m=\ell+1}^{n}\frac{1}{\sqrt{m}}\Bigr)\\
	&\geq \sum_{\ell=n-C_6\sqrt{n}}^{n}C_3\ell^{2j-2}\Bigr(1-2C_5\frac{C_6\sqrt{n}}{\sqrt{n-C_6\sqrt{n}}}\Bigr)\\
	&\geq \sum_{\ell=n-C_6\sqrt{n}}^{n}C_3\ell^{2j-2}(1-4C_5C_6)\\
	&\geq C_3(1-4C_5C_6)C_6\sqrt{n}\Bigr(\frac{n}{2}\Bigr)^{2j-2}\\
	&\geq C_7n^{2j-3/2},
	\end{align*}
	for some $C_7>0$. Because $v_n$ is a polynomial in $n$ for $n\geq 2(k-j)$ (see Lemma \ref{lemma_poly}), there exists $C>0$ and $n_0>0$ such that for $n\geq n_0$, $v_n\geq Cn^{2j-1}$.
\end{proof}

\acknowledgements
The author is grateful for the collaboration with Mathilde Bouvel and Valentin F\'eray who deeply supported the development of this work.

The author thanks Adrian R\"ollin for discussions on Stein's method and for pointing out useful references. Appreciated were also the constructive comments from Larry Goldstein and from the anonymous referees.

\nocite{*}
\bibliographystyle{abbrvnat}
\bibliography{biblio}

\end{document}